\documentclass[12pt]{article}%
   
\usepackage{amsmath,enumerate}
\usepackage{amsfonts}
\usepackage{amssymb}

\usepackage{tikz}

\usepackage{lipsum}
\usepackage{pgfplots}
\pgfplotsset{colormap={coloronemap}{rgb=(.4,.4,1); rgb=(.8,.8,1)}}
\pgfplotsset{colormap={colortwomap}{rgb=(1,.4,.4); rgb=(1,.8,.8)}}
\usepackage{eso-pic,calc}
\usepackage[font=small]{caption}
\usepgfplotslibrary{external}
\usetikzlibrary{calc}
\usetikzlibrary{shadings}
\usepgflibrary{shapes.geometric}

\pgfplotsset{my style/.append style={axis x line=middle, axis y line=middle, xlabel={$x$}, ylabel={$y$}, axis equal}}

\newtheorem{theorem}{Theorem}[section]

\newtheorem{claim}[theorem]{Claim}

\newtheorem{conjecture}[theorem]{Conjecture}

\newtheorem{defin}[theorem]{Definition}

\newtheorem{lemma}[theorem]{Lemma}

\newtheorem{proposition}[theorem]{Proposition}

\newtheorem{question}[theorem]{Question}

\newenvironment{proof}[1][Proof]{\noindent\textbf{#1.} }
{\hfill \ \rule{0.5em}{0.5em}}

\begin{document}

\title{Triangles in $K_s$-saturated graphs with minimum degree $t$}
\author{Benjamin Cole\thanks{Department of Mathematics and Statistics, 
California State University, Sacramento, \texttt{benjamincole@csus.edu}} 
\and 
Albert Curry\thanks{Department of Mathematics and Statistics,
California State University, Sacramento, \texttt{albertcurry@csus.edu}}
\and 
David Davini\thanks{Department of Mathematics and Statistics,
University of California, Los Angeles, \texttt{daviddavini@ucla.edu}}
\and 
Craig Timmons\thanks{Department of Mathematics and Statistics, California State University, Sacramento.
Research is supported by a grant from the Simons Foundation \#359419.}}
\date{}

\maketitle

\vspace{-2em}

\begin{abstract}
For $n \geq 15$, we prove that the minimum number of triangles in an $n$-vertex 
$K_4$-saturated graph with minimum degree 4 is exactly $2n-4$, and that there is a unique extremal
graph. 
This is a triangle version of a result of 
Alon, Erd\H{o}s, Holzman, and Krivelevich from 1996.
 Additionally, we show that for any $s > r \geq 3$ and $t \geq 2 (s-2)+1$, there
 is a $K_s$-saturated $n$-vertex graph with minimum degree $t$ that has 
 $\binom{ s-2}{r-1}2^{r-1} n + c_{s,r,t}$ copies of $K_r$.  This shows that 
 unlike the number of edges, the number 
 of $K_r$'s ($r >2$) in a $K_s$-saturated graph 
 is not forced to grow with the minimum degree, except for possibly in lower order terms. 
\end{abstract}


\section{Introduction}

Let $F$ be a graph.  A graph $G$ is \emph{$F$-free} if $G$ does not contain $F$ as a subgraph.  
A graph $G$ is \emph{$F$-saturated} if $G$ is $F$-free, 
and adding a new edge to $G$ creates a copy of $F$.
The minimum number of edges in an $F$-saturated graph with 
$n$ vertices is called the \emph{saturation number of $F$}.
Write $\textup{sat}(n , F)$ for this minimum, so 
\[
\textup{sat}(n ,F) = \min \{ | E(G) | : \mbox{$G$ has $n$ vertices and is $F$-saturated} \}.
\]
An $n$-vertex $F$-saturated
graph with $\textup{sat}(n,F)$ edges is called an \emph{extremal graph}.  


\subsection{History and Previous Results}

One of the most important results on graph saturation 
is that for any graph $F$ with at least one edge, there is a constant $C_F$ such that 
$\textup{sat}(n , F) \leq C_F n$.  This was proved by K\'{a}szonyi and Tuza in 1986 \cite{kt},
and shows that saturation numbers are linear in $n$.      
Since then, the study of saturation has become an established branch 
of extremal graph theory.   
Saturation numbers of hypergraphs and of random graphs have been studied 
as well \cite{bol65, ks, pikhurko, pikhurko2}.
The survey paper of Faudree, Faudree, and Schmitt \cite{sat survey} contains 
many results and references. 

For complete graphs, Erd\H{o}s, Hajnal, and Moon \cite{ehm} proved that for $s \geq 3$,
\[
\textup{sat}(n , K_s) = (s-2) ( n - s +2) + \binom{s-2}{2}.
\]
Furthermore, there is a unique extremal graph which is, up to isomorphism, 
$K_{s-2} + \overline{K_{n - s + 2} }$ (the join of a clique with $s-2$ vertices
and an independent set with $n-s+2$ vertices).
This graph has 
minimum degree 
$s-2$, and the minimum degree of a $K_s$-saturated graph is at least $s-2$ (since nonadjacent vertices must have 
a $K_{s-2}$ in their common neighborhood).  
A natural question
is to ask for the minimum number of edges in an $F$-saturated graph $G$
with $\delta (G) = t$ where $t > s-2$.  
Given a graph $F$ and 
an integer $t$, let
\[
\textup{sat}_t ( n ,F) = 
 \min \{ | E(G) | : \mbox{$G$ has $n$ vertices, is $F$-saturated, and has $\delta (G) = t$ } \}.
\]
Duffus and Hanson \cite{dh} proved that $\textup{sat}_2 (n , K_3) = 2n -5$ for $n \geq 5$, and  
characterized the extremal graphs. They also showed that 
$\textup{sat}_3 (n , K_3) = 3n -15$ for $n \geq 10$.  
For larger $t$ and $s$, the bounds are not exact.  
In 2014, Day \cite{day}, resolving a conjecture of Bollob\'{a}s \cite{graham} from 1996, 
showed that 
\begin{equation}\label{equation day}
\textup{sat}_t  (n , K_s) \geq t n - c_{t}
\end{equation}
for any $s \geq 3$, $t \geq s -2$.  Here $c_{t}$ is a constant depending only on $t$.
Further discussion on saturated graphs with degree constraints can be found 
in \cite{sat survey}. 

In 2014,  Alon and Shikhelman \cite{as}
introduced a very important generalization of Tur\'{a}n numbers which has since 
been extensively studied (\cite{ergemlidze, gerbner, shapira, ma} to name a few).  
It is connected to the widely studied Tur\'{a}n problem
for Berge hypergraphs (\cite{gmv, pttw}, for instance).  Motivated by this generalization,   
Kritschgau et.\ al.\ \cite{kmtt} defined an analogous generalization of saturation numbers. 
For graphs $H$ and $F$, let
$\textup{sat}(n , H , F)$
be the minimum number of copies of $H$ in an $F$-saturated graph
with $n$ vertices. Observe that $\textup{sat}(n , K_2 ,F) = \textup{sat}(n , F)$.
In this paper, we introduce the function
\[
\textup{sat}_t (n , H , F)
\]
which is defined to be the minimum number of copies of $H$ in an $n$-vertex 
$F$-saturated graph with minimum degree $t$.  
This function generalizes both $\textup{sat}_t (n,F)$ and $\textup{sat}(n, H , F)$.
The question we put forth is the following.

\begin{question}\label{question 1}
Let $H$ and $F$ be graphs.  What is the minimum number 
of copies of $H$ in an $F$-saturated $n$-vertex graph with minimum degree
$t$?  
\end{question}
 
Before stating our results, let us recall a result from \cite{kmtt} 
which we take as a starting point for our work.  
In \cite{kmtt}, the formula  
\begin{equation}\label{exact bound k3 k4}
\textup{sat}(n,K_3,K_4) = n-2
\end{equation}
was proved, and it was shown that $K_2 + \overline{K_{n-2}}$ is the unique extremal graph.  
 Bounds on $\textup{sat}(n,K_r,K_s)$ for 
 all $s > r \geq 3$ were also proved in \cite{kmtt}.  However, 
 an asymptotic formula for $\textup{sat}(n,K_r,K_s)$ is known only
 in the case when $(r,s) = (3,4)$.    

Motivated by the fact that (\ref{exact bound k3 k4}) gives a formula for
 $\textup{sat}(n , K_3 , K_4)$, 
 we will study the function $\textup{sat}_t (n , K_3 , K_4)$ in detail.  
 Our aim is to determine whether or not a statement similar to (\ref{equation day}) holds
 when counting $K_3$'s in a $K_4$-saturated graph with minimum degree at least $t$.  First
let us state some results that answer Question \ref{question 1}
in certain cases.  The first proposition we give is easy to prove using the fact
that in a $K_s$-saturated graph, any two nonadjacent vertices must have 
a $K_{s-2}$ in their common neighborhood.

\begin{proposition}\label{proposition 1}
Let $s > r  \geq 2$ be integers.  If $G$ is an $n$-vertex $K_s$-saturated graph with 
$\delta (G) = s-2$, then $G$ is isomorphic to  $K_{s-2} + \overline{ K_{n-s+2}}$ and consequently, has
\[
\binom{s-2}{r} + (n-s+2) \binom{s-2}{r-1}
\]
copies of $K_r$.  
\end{proposition}

In the special case $(r,s)  = (3,4)$ and $\delta (G) = 2$, we have exactly $n-2$ triangles.
A close look at the proof of (\ref{exact bound k3 k4}) in \cite{kmtt} shows that 
if $G$ is not isomorphic to $K_2 + \overline{K_{n-2}}$, then $G$ has at 
least $n$ triangles.  Now if $G$ is $K_4$-saturated and is not isomorphic to 
$K_2 + \overline{K_{n-2}}$, then one can prove that $G$ 
has minimum degree at least 3.  
Thus, we see a small jump in the number of triangles when $G$ is no longer allowed 
to have a vertex of degree 2.  It turns out that when
the minimum degree is 3, we do not just get 2 additional triangles, but we must get
and additional $n - 5$ triangles.
Before formalizing this as a proposition, we need to introduce a new graph.  

Let $a_1a_2a_3a_4a_5 a_1$ be a cycle of length 
five and let $b$ be a vertex joined to each $a_i$.  Call this graph $W$. 
Let $m_1,m_3,m_4$ be positive integers 
with $m_1 +m_3 + m_4 = n - s  +1$.  In $W$, replace $b$ with a clique of size $s-3$, and for
$i \in \{1,3,4 \}$, replace $a_i$ with an independent set of size $m_i$.  Vertices in this new graph 
are adjacent if and only if the vertices they replaced are adjacent vertices in $W$.  
Write $W_s(m_1 , 1 , m_3 , m_4 , 1)$ for this graph, which has appeared in the literature (see \cite{afg}).  
 
\begin{proposition}\label{proposition 2}
Let $s  \geq 3$.  If $G$ is an $n$-vertex $K_s$-saturated  
graph with $\delta (G) =s-1$, then $G$ is isomorphic to either $(K_{s-1} -e ) + \overline{K_{n-s+1}}$,
or $W_s (m_1 , 1 , m_3,m_4,1)$ for some $m_1,m_3,m_4$ with $m_1 + m_3 + m_4 = n-s+1$.
\end{proposition}

It follows from Proposition \ref{proposition 2} that if $G$ is an $n$-vertex $K_4$-saturated graph
with minimum degree 3, then $G$ has at least $2n-7$ triangles, and equality holds
only if $G$ is isomorphic to $W_4(m_1,1,m_3,1,1)$ for some
$m_1 + m_3 = n -4$.  For a proof of Proposition \ref{proposition 2}, see \cite{kmtt}.  


\subsection{New Results}

In light of Day's Theorem (\ref{equation day}) and that
\begin{itemize}
\item $|V(G)| = n $, $G$ is $K_4$-saturated, and $\delta (G) = 2$ ~~ $\Rightarrow$ ~~ $G$ has at least $n-2$ triangles,
\item $|V(G)| = n $,~$G$ is $K_4$-saturated, and $\delta (G) = 3$ ~~ $\Rightarrow$ ~~ $G$ has at least $2n-7$ triangles,
\end{itemize}
one may be tempted to conjecture that in general, $\delta (G) = t$ forces at least $(t-1)n - O(1)$ triangles in 
any $n$-vertex $K_4$-saturated graph.  This would then give a version of
 Day's result for triangles.

It turns out, rather surprisingly, that
for any $t \geq 4$ and $n \geq t +5$, there is an $n$-vertex $K_4$-saturated graph 
that has minimum degree $t$ and only $2n+2t-12$ triangles.
We call this graph $H_t(n)$ and it is defined in Section \ref{section 2}. 
Our main theorem determines $\textup{sat}_t (n , K_3 , K_4)$, and shows $H_t(n)$ is the unique extremal graph.

\begin{theorem}\label{main theorem}
Let $n \geq 14$ and $G$ be an $n$-vertex $K_4$-saturated graph
with $\delta (G) = 4$.
Then $G$ contains at least $ 2n -4$ triangles.  
Furthermore, if $G$ contains exactly $ 2n-4$ triangles, 
then $G$ is isomorphic to $H_t(n)$.
\end{theorem}

For $t \geq 4$, the graph $H_t (n)$ implies the following upper bound on $\textup{sat}_t (n , K_3 , K_4)$.

\begin{theorem}\label{H proposition}
For integers $t \geq 4$ and $n \geq 2t$, 
\[
\textup{sat}_t (n , K_3 , K_4) \leq 2n + 2t -12.
\]
\end{theorem}

We conjecture that the upper bound in Theorem \ref{H proposition} is best possible, and that 
$H_t (n)$ is the unique extremal graph for all sufficiently large $n$. 

\begin{conjecture}\label{conjecture 1}
For any integer $t \geq 4$, there is an integer $n_t$ such that for all $n \geq n_t$,
\[
\textup{sat}_t (n , K_3 , K_4 ) = 2n + 2t -12
\]
and $H_t(n)$ is the unique extremal graph.
\end{conjecture}

Theorem \ref{main theorem}
shows that Conjecture \ref{conjecture 1} is correct when $t=4$.
The next theorem gives an upper bound for arbitrary $s > r \geq 3$.

\begin{theorem}\label{3 against s}
Let $s  > r \geq  3 $ and $t \geq 2 (s-2)+1$ be integers.  For $n \geq 2(s-2) + 2t$,
\[
\textup{sat}_t (n , K_r , K_s ) \leq \binom{s-2}{r-1} 2^{r-1} n + C_{s,r,t}
\]
where $C_{s,r,t}$ is a constant depending only on $r$, $s$, and $t$.  
\end{theorem}

It would be interesting to determine if there is a constant $c_{r,s}$, depending only on $r$ and $s$, such that 
for all $n \geq n(r,s,t)$ and $t \geq s-2$, there is an $n$-vertex $K_s$-saturated graph with minimum degree $t$
having at most $c_{r,s} n + o(n)$ copies of $K_r$.   Theorem \ref{H proposition} shows that 
such a constant exists in the case when $(r,s) = (3,4)$.  
Theorem \ref{3 against s} covers all $s > r \geq 3$, but assumes $t \geq 2 (s-2) +1$, and perhaps
the coefficient of $n$ could be improved (as in Theorem \ref{H proposition} 
when $(r,s)= (3,4)$).  We can improve this upper bound in the case $(r,s) = (3,5)$ as the 
following result shows.

\begin{theorem}\label{3 against 5}
For any $t \geq 8$ and $n \geq t + 30$,
\[
\textup{sat}_t ( n , K_3 , K_5) \leq 9n + C_t
\]
where $C_t$ is a constant dependent only on $t$.   
\end{theorem}

Finally, a simple argument in the case $r =3$ gives a lower bound on $\textup{sat}_t (n , K_3 , K_s)$.   

\begin{proposition}\label{lb 3}
Let $s > 3 $ and $t \geq 6 \binom{s-2}{2}$ be integers.  For 
$n \geq 2s - 2$,
\[
\binom{s-2}{2}(n-2) \leq \textup{sat}_t (n , K_3 , K_s).
\]  
\end{proposition}

Proposition \ref{lb 3} implies $\textup{sat}_t ( n , K_3 , K_5) \geq 3 (n -2)$ for $t \geq 18$.  We suspect 
the upper bound of Theorem \ref{3 against 5} could give the correct coefficient of $n$.

In Section \ref{section 2}, we define the graphs that prove 
Theorems \ref{H proposition}, \ref{3 against s}, 
and Proposition \ref{3 against 5}.
In Section \ref{section proof}, we prove Theorem \ref{main theorem}. 
In Section \ref{lb 3 section}, we prove Proposition \ref{lb 3}.


\section{Proof of Theorems \ref{H proposition},
\ref{3 against s}, and \ref{3 against 5}}\label{section 2}


A natural question one might ask is whether or not every edge $e$ in a $K_s$-saturated graph
is in a triangle.  It turns out that this is not necessarily true. 
In fact, $K_s$-saturated graphs that have an edge in no triangle
contain certain subgraphs--what we call \textit{$K_s$-support structures}--which pave the way for constructions proving Theorems \ref{H proposition} and \ref{3 against s}. For this and the following sections, we will use $k_{r}(A)$ to denote the number of copies of $K_{r}$ in $A$.

\begin{defin}\label{support-structures}
Let $S(A,B)$ be a graph with vertex set $A \cup B$, where $A$ and $B$ are disjoint sets. We call $S(A,B)$ a \textbf{pre-$K_s$-support-structure} provided:
\begin{itemize}
\item $A$ and $B$ are both $K_{s-1}$-free,
\item $N(a)\cap B$ has a copy of $K_{s-2}$ for every vertex $a\in A$, 
\item $N(b)\cap A$ has a copy of $K_{s-2}$ for every vertex $b\in B$, and
\item $S(A,B)$ itself is $K_s$-free.
\end{itemize}
If $S(A,B)$ has the added condition that 
the subgraph induced by $A$ and the subgraph induced by $B$ are $K_{s-1}$-saturated, 
and adding any missing edge between $A$ and $B$ would create a $K_s$, then we call $S(A,B)$ a \textbf{$K_s$-support-structure}.
\end{defin}
\begin{lemma}\label{pre-to-regular}
Given a pre-$K_s$-support-structure $S(A,B)$, there is a  $K_s$-support-structure $S'(A,B)$ where $S'(A,B)$ is a supergraph of $S(A,B)$.
\end{lemma}
\begin{proof}
We add edges one at a time to $S(A,B)$ arbitrarily, adding each missing edge $e$ only if it does not create a $K_{s-1}$ in $A$, a $K_{s-1}$ in $B$, or a $K_s$ in $S(A,B)$.
We continue this process until there are no longer any edges to add. Certainly after at most $\binom{|A\cup B|}{2}$ steps this process will end. Notice that with each edge added, the new graph is still a pre-$K_s$-support-structure, and in fact the final graph $S'(A,B)$ is a $K_s$-support-structure.
\end{proof}
\begin{lemma}\label{support-construction}
Let $S(A,B)$ be a $K_s$-support-structure, and set $N=t-\min\{|B|, \delta_{S(A,B)}(A)\}$ and $M=t-\min\{|A|, \delta_{S(A,B)}(B)\}$. Given integers 
\[
t > \min\{|A|, |B|, \delta_{S(A,B)}(A), \delta_{S(A,B)}(B)\} ~~\mbox{and}~~ 
n > |A|+|B|+N+M,
\]
there exists a $K_s$-saturated supergraph $G$ of $S(A,B)$ with order n and minimum degree t containing 
$n k_{r-1}(A) + C_{|A|,|B|,t}$ copies of $K_r$.
\end{lemma}
\begin{proof} Define $G$ as the graph obtained from $S(A,B)$ by 
taking two independent sets $X=\{x_1, x_2,\dots , x_{n-|A|-|B|-M}\}$ 
and  $Y=\{y_1, y_2,\dots , y_M\}$, and joining $X$ to $A \cup Y$, and joining 
$Y$ to $B \cup X$. Note that since $S(A,B)$ is a $K_s$-support-structure, $G$ is $K_s$-saturated. Next we show that the minimum degree of G is $t$. We have 
\begin{eqnarray*}
    \delta_G(Y\cup A) &= & \min\{\delta_G(Y), \delta_G(A)\} =|X| + \min\{|B|, \delta_{S(A,B)}(A)\} \\
    &\geq & N + \min\{|B|, \delta_{S(A,B)}(A)\} =t .
\end{eqnarray*}
By symmetry, $\delta_G ( X \cup B) \geq t$.  
Putting these two inequalities together, we see that $\delta(G)=\min\{\delta_G(Y\cup A), \delta_G(X\cup B)\}=t$.

We now count the number of $K_r$'s in $G$.  Note that by $G$'s construction, the number of vertices not in $X$ is dependent 
only on $|A|$, $|B|$ and $t$. Thus, the number of copies of $K_r$ that do not use a vertex in $X$ is also dependent 
only on these parameters.  
Suppose $K$ is a $K_r$ that does contain a vertex in $X$. Then $K$ can only contain one vertex in $X$, since $X$ is an independent set. Also, $K$ cannot contain any vertex in $Y$ because $r \geq 3$ and $Y$ is an independent set. Thus, all the vertices in $K$ that are not in $X$ must be in $A$. We conclude that the number of copies of $K_r$ containing a vertex in $X$ is \[
|X| k_{r-1}(A) = ( n - |A| - |B| - M ) k_{r-1}(A).  
\]
\end{proof}

We now define three graphs by defining pre-$K_s$-support structures, and then
applying Lemmas \ref{pre-to-regular} and \ref{support-construction}. 

\smallskip

\noindent
\textbf{The graph $H_t(n)$:} Let $t\geq 4$ and $n > 2t$ be integers. Let $A = \{ a_1,a_2,a_3,a_4 \}$ and $B = \{ b_{123} , b_{124} , b_{134} , b_{234} \}$, and construct the graph $S(A,B)$ with vertex set $A \cup B$, and the following edges:
\begin{itemize}
\item $a_1 a_2$ and $a_3  a_4$,
\item $b_{123} b_{234}$, $b_{234} b_{124}$, $b_{124} b_{134}$, and $b_{134} b_{123}$,
\item $a_r b_{ijk}$ if and only if $r \in \{i,j,k \}$.
\end{itemize} 
It is easy to verify that $S(A,B)$ is a $K_s$-support-structure. Employing Lemmas \ref{pre-to-regular} and \ref{support-construction}, 
we obtain a $K_4$-saturated graph, which we call $H_t(n)$, that has $n$-vertices, 
minimum degree $t$, and containing $2n +2t -12$ triangles. 
Thus, $H_t(n)$ proves Theorem \ref{H proposition}.

\smallskip
\noindent
\textbf{Remark:} The graph $H_4(n)$ appears briefly in \cite{aehk}, but the focus in that paper is on minimizing the number of edges in a $K_4$-saturated graph with minimum degree 4.  Up to isomorphism, there is a unique graph that minimizes the number of edges in a $K_4$-saturated graph with minimum degree 4, but this graph does not minimize the number of triangles.

\bigskip

Next we construct the graph that proves Theorem \ref{3 against s}.   

\smallskip

\noindent
\textbf{The graph $F_{s,t}(n)$:} Let $s > r \geq 3$, $t \geq 2 ( s-2)+1$, and $n \geq 2(s-2) + 2t$. Let $A = \{a_1 , a_2, \dots , a_{2(s-2)} \}$ and $B = \{ b_1 , b_2 , \dots , b_{ 2 (s-2) } \}$, and construct the graph $S(A,B)$ with vertex 
set $A \cup B$, and the following edges:
\begin{itemize}
\item $a_i a_j$ if and only if $i \not\equiv j+(s-2) \pmod{2(s-2)}$,
\item $b_i b_j$ if and only if $i \not\equiv j+(s-2) \pmod{2(s-2)}$,
\item $a_i b_{i \pmod{2(s-2)}}$, $a_i b_{i+1 \pmod{2(s-2)}}$, \dots , and $a_i b_{i+(s-2) -1 \pmod{2(s-2)}}$.
\end{itemize}
Notice that the vertices in $A$ and the vertices in $B$ form two complete $(s-2)$-partite graphs with 
2 vertices in each part.  

We now show that $S(A,B)$ is a pre-$K_s$-support-structure. 
It is easy to verify that $A$ and $B$ are $K_{s-1}$-free, and that each $a \in A$ and $b \in B$ has a copy of $K_{s-2}$ in $N(a)\cap B$ and $N(a)\cap B$, respectively. 
Next we prove that $S(A,B)$ is $K_s$-free.  
It is enough to show that $a_1$ does not lie in a $K_s$ since $S(A,B)$ is vertex transitive.
The neighborhood of $a_1$ is the union of the two sets 
\begin{center}
$\mathcal{A}_1 = \{a_2 , a_3 , \dots , a_{ 2 (s-2) } \} \backslash \{ a_{s-1} \}$
~~and~~
$\mathcal{B}_1 = \{ b_1 , b_2 , \dots , b_{s-2} \}$.
\end{center}
Let $H_1$ be the subgraph of $S(A,B)$ induced by $\mathcal{A}_1 \cup \mathcal{B}_1$.  
We will show $H_1$ is $K_{s-1}$-free.  
It 
will be convenient to relabel the vertices in $\mathcal{A}_1$ as
\begin{center}
$a_s=1$, $a_{s+1} = 2$, $a_{s+2} = 3, \dots , a_{2 (s-2)} = s-3$,
\end{center}
and
\begin{center}
 $a_2 = s-2$, $a_3 = s-1, a_4  =  s , \dots , a_{s-2}=  2s-6$.
 \end{center}
With this relabeling, in the graph $H_1$ the vertex $b_i$~($1 \leq i \leq s-2$) is adjacent to all vertices in
\begin{equation}\label{b adjacencies}
\{ i , i+1 , i+2 , \dots , i + s- 4 \} \cup ( \mathcal{B}_1 \backslash \{ b_i \} ).
\end{equation}
Now suppose $K$ is a clique in $H_1$.  Let $b_{i_1} , b_{i_2} , \dots , b_{i_{ \beta} }$ be the vertices in $K$ 
that are in $\mathcal{B}_1$, where $1 \leq i_1 < i_2 < \dots < i_{ \beta} \leq s- 2$.  
By (\ref{b adjacencies}), if $a \in \{1,2, 3 , \dots , 2s-6 \}$ is a vertex in $\mathcal{A}_1 \cap K$, then
\[
i_{ \ell } \leq a \leq i_{ \ell} + s- 4
\]
for $\ell = 1 ,2 , \dots , \beta$.  In particular, $i_{ \beta} \leq a \leq i_1 + s - 4$.  
Adding $- i_{ \beta} +1$ throughout leads to
\begin{eqnarray*}
1 & \leq & a - i_{ \beta} + 1 \leq i_1 - i_{ \beta} + 1 + s -4 = - ( i_{\beta} - i_1 )   + s -3 \\
& \leq & - \beta +1 + s- 3 = s - \beta - 2.
\end{eqnarray*}
Therefore, there are at most $s - \beta - 2$ vertices in $\mathcal{A}_1$ that are in $K$, so $K$ contains at most
\[
( s- \beta -2 ) + \beta = s-2
\]
vertices.  We conclude that the neighborhood of $a_1$ is $K_{s-1}$-free which implies $S(A,B)$ is $K_s$-free.

Adding edges to $S(A,B)$ using the process described in Lemma \ref{pre-to-regular}, we 
obtain a  $K_s$-support-structure $S'(A,B)$. Notice that since $A$ and $B$ were already $K_{r-1}$-saturated, the only edges added were  between $A$ and $B$, so the subgraphs induced by $A$ and $B$ have not changed. Applying Lemma \ref{support-construction} to $S'(A,B)$, we obtain a $K_s$-saturated graph $F_{s,t}(n)$ with $n$ vertices, minimum degree $t$, and
containing at most $nk_{r-1}(A)+C_{r,s,t} = \binom{s-2}{r-1}2^{r-1}n+C_{r,s,t}$ copies of $K_r$.

A computation, using the assumption that $n \geq 2(s-2) + 2t$, shows that all vertices in $F_{s,t}(n)$ have 
degree at least $t$ (vertices in $X$ will have degree $t$).  The graph $F_{s,t}(n)$ proves Theorem \ref{3 against s}.  

\bigskip

\textbf{The graph $R_{t}(n)$:} Let $t > 9$, and $n > 2t+12$. Let $A$ be the disjoint union of the three $K_{3}$'s
\[ 
A_1  = \{a_{147}, a_{258}, a_{369}\}, ~~ ,  A_2 = \{a_{159}, a_{267}, a_{369}\}, ~~
\mbox{and} ~~
    A_3 = \{a_{169}, a_{249}, a_{357}\}.
    \]
The subscripts appearing in a given triangle form a class of parallel lines in an affine plane of order 3.  

Next, let $B$ be the disjoint union of triangles $B_1$ through $B_9$, where each $B_i$ is given by
\begin{align*}
    B_i = \{b_{i,1}, b_{i,2}, b_{i,3}\}.
\end{align*}
Construct the graph $S(A,B)$ by connecting each vertex $a_{ijk} \in A_s$ to all vertices in the union $B_i\cup B_j\cup B_k\cup \{b_{t,s}:1\leq t\leq 9\}$. 

We now show that $S(A,B)$ is a pre-$K_5$-support structure. It is easy to verify that $A$ and $B$ are $K_{4}$-free, and that each $a \in A$ and $b \in B$ has a copy of $K_{3}$ in $N(a)\cap B$ and $N(a)\cap B$, respectively. Also, since $A$ and $B$ are $K_4$-free, a possible $K_5$ in $S(A,B)$ must have at least two vertices in $A$, say $a_{i_1j_1k_1},a_{i_2j_2k_2}\in A_s$. However, by $S(A,B)$'s construction, their common neighborhood in $B$ contains at most 1 vertex. Hence $S(A,B)$ is $K_5$-free.

Adding edges to $S(A,B)$ using the process described in Lemma \ref{pre-to-regular}, we
obtain a  $K_s$-support-structure $S'(A,B)$. Importantly, since by $S(A,B)$'s construction every pair of non-adjacent vertices in $A$ must have a $K_{3}$ in their common neighborhood in $B$, no edges were added to $A$, so the subgraph induced by $A$ has not changed. Applying Lemma \ref{support-construction} to $S'(A,B)$, we obtain a $K_5$-saturated graph $R_{t}(n)$ with $n$ vertices, minimum degree $t$, and
containing at most $nk_{3}(A)+C_{t}= 9n+C_{t}$ copies of $K_3$.
\bigskip

\noindent
\textbf{Remark:} The construction of $R_{t}(n)$
can be generalized to create a $K_s$-support-structure $R_{s,t}(n)$ 
for any $s\geq 3$, but for brevity's sake we have chosen to omit this. Notice that for $s=5$, $R_{5,t}(n)$'s $9n+C_t$ copies of $K_3$ is less than 
 $F_{5,t}(n)$'s $12n+C_t$ copies of $K_3$. This turns out to be an exception; 
 for $s=6$, $R_{6,t}(n)$ and $F_{6,t}(n)$ both have $24n+C_t$ copies of $K_3$, and for $s\geq7$, $R_{7,t}(n)$ contains more copies of $K_3$ than $F_{7,t}(n)$. In general, $R_{s,t}(n)$ will
 have fewer copies of $K_r$ than $F_{s,t}(n)$ for $s \in \{4,5, \dots, 2^{r-1}+1\}$, but when $s > 2^{r-1}+2$, $R_{s,t}(n)$ contains more copies of $K_r$ than $F_{s,t}(n)$. This suggests that a better construction may exist.  


\section{Proof of Theorem \ref{main theorem}}\label{section proof}

In this section, we prove Theorem \ref{main theorem}.  The proof will be broken up into several
lemmas.   
We write $k_3(G)$ 
for the number of triangles in $G$.  For $x \in V(G)$, $N(x)$ is the neighborhood of $x$.  
Let us set up some specialized notation that will be used in the rest of this section.  
This same notation was used in \cite{aehk}.  

\bigskip

\noindent
\textbf{Setup and Notation:} 
For the remainder of Section \ref{section proof}, 
let $G$ be an $n$-vertex graph that is 
$K_4$-saturated, and has
$\delta (G) = 4$.
Let $x \in V(G)$ be a vertex of degree 4, and let 
\[
N(x) = \{ x_1,x_2,x_3,x_4 \} ~~~ \mbox{and} ~~~ N[x] = \{ x \} \cup N(x).
\] 
Let $Y = V(G) \backslash N[x]$.
If $y \in Y$, then $N(y) \cap N(x)$ must contain an edge so
that every vertex in $Y$ is adjacent to 
at least two vertices in $N(x)$.
For $S \subseteq \{1,2,3,4 \}$, let  $V_S$ be the vertices 
$y \in Y$ such that $y$ is a neighbor of $x_i$ if and only if $i \in S$.
For notational convenience, we will omit braces and commas so if, say
$S = \{1,3,4 \}$, then we write $V_{134}$ rather than
$V_{ \{ 1,3 ,4 \} }$.  
Given sets $S, T \subseteq \{1,2,3,4 \}$, write 
\[
V_S \sim V_T
\]
if all vertices in $V_S$ are adjacent to all vertices in $V_T$,
and $V_S \nsim V_T$ if no vertex in $V_S$ is adjacent to a vertex
in $V_T$.  

An important observation is that if $V_S \neq \emptyset$, then there must be a pair $\{i,j \} \subseteq S$ such
that $x_i$ is adjacent to $x_j$.  Also,
if $\{x_i : i \in S \}$ forms an independent set in $N(x)$, then $V_S = \emptyset$.  
Because of this, \underline{we  ignore all $V_S$} 
for which $\{x_i : i \in S \}$ is an independent set in $N(x)$.  

Given a subset $S \subset \{1,2,3,4 \}$ with $i \notin S$ for some $i \in \{1,2,3,4 \}$, 
let $\mathcal{T}_{S,i}$ be the set 
of all $T \subseteq \{1,2,3,4 \}$ for which $i \in T$, and for any $j,k \in S$ for which $x_j$ and $x_k$ are adjacent
vertices in $N(x)$, we have $| T \cap \{j,k \} | \leq 1$.  

\begin{lemma}[Rules Lemma]\label{general lemma}
Let $S \subset \{1,2,3,4 \}$ with $i \notin S$ for some $i \in \{1,2,3,4 \}$.
If $V_S \neq \emptyset$, then $y$ is adjacent to some vertex in 
\[
\bigcup_{T \in \mathcal{T}_{S,i} } T
\]
and in particular, this union is not empty.
\end{lemma}
\begin{proof} 
Suppose $y \in V_S$.  Since $i \notin S$, the vertex $y$ is not adjacent to $x_i$ where $x_i \in N (x)$.  
There must be an edge $\alpha \beta $ with both $ \alpha $ and $\beta$ in the intersection $ N(y) \cap N(x_i)$.
If $\alpha , \beta \in N(x)$, then $\{ x , x_i , \alpha , \beta \}$ is a $K_4$ which is a contradiction.  
Thus, we may assume that $\alpha \in Y$, say $\alpha \in V_T$.  Because $\alpha$ is adjacent to 
$x_i$, we have $i \in T$.  If there is a $j,k \in S$ for which $x_j x_k $ is an edge in $N(x)$ 
and $j,k \in T$, then $\{ y , \alpha , x_j , x_k \}$ is a $K_4$.  
We conclude that $T$ must be some member of $\mathcal{T}_{S,i}$.  
\end{proof}

\bigskip

It will be convenient to represent Lemma \ref{general lemma} with an arrow diagram and 
we refer to these as Rules.  Assuming the set up and conclusion of Lemma \ref{general lemma} where
$\mathcal{T}_{S,i} = \{ T_1 , T_2 , \dots , T_{ \ell} \}$, then we illustrate 
using the figure below.  
\begin{center}

\begin{tikzpicture}[scale=.8]

\draw [->] (-.1,1.5) -- (2,0);
\draw [->] (-.1,1.5) -- (2,2);
\draw [->] (-.1,1.5) -- (2,3);

\draw (-.1,1.5) node[left]{$V_{S}$};
\draw (2,0) node[right]{$V_{T_{ \ell} }$};
\draw (2,2) node[right]{$V_{T_2}$};
\draw (2,3) node[right]{$V_{T_1}$};

\draw(2.2,1.1) node[right]{$\vdots$};

\end{tikzpicture}

Rule for $V_S$ and choice of $i \notin S$
\end{center}
We also use the notation $V_S \rightarrow V_{T_1} \cup V_{T_2} \cup \dots \cup V_{T_{ \ell } }$.  
Note that the assertion of Lemma \ref{general lemma} is not that $T_j \neq \emptyset$ 
for every $T_j \in \mathcal{T}_{S,i}$, but 
only that at least one of these sets is not empty.  

The first application of Lemma \ref{general lemma} will be 
in the proof of Lemma \ref{claim 3 edges}.
First we prove two easier lemmas.  

\begin{lemma}\label{lemma 1 edge}
The number of edges in $N(x)$ is at least 2.
\end{lemma}
\begin{proof}[Proof of Lemma \ref{lemma 1 edge}]
See Lemma 3.2 of \cite{kmtt} (the proof is not difficult).
\end{proof}

\begin{lemma}\label{lemma 2 edges}
If $n \geq 14$ and the number of edges in $N(x)$ is 2, then 
\[
k_3 (G) \geq 2n -4.
\]
Furthermore, equality holds if and only if $G$ is isomorphic to $H_4(n)$.
\end{lemma}
\begin{proof}[Proof of Lemma \ref{lemma 2 edges}]
There are two possibilities for the edges in $N(x)$.  

\medskip
\noindent
\textit{Case 1:} The two edges in $N(x)$ form a path.
\smallskip

Suppose the edges in $N(x)$ are $x_1 x_2$ and $x_2 x_3$.  
Every vertex in $Y$ must be adjacent to $x_2$ since $N(x) \cap N(y)$ must contain
an edge for all $y \in Y$, and every edge in $N(x)$ has $x_2$ as an endpoint.  
Therefore, $x_2$ is adjacent to all vertices in $G$ except for $x_4$.  
The intersection $N(x_3) \cap N(x_4)$ must contain an edge, 
say $y_1 y_2$.  The edge $y_1 y_2$ must have both endpoints
 in $Y$ since $x$ is the only neighbor of $x_4$ in $N[x]$, 
but then $\{y_1,y_2,x_2 ,x_3 \}$ is a 
 $K_4$ in $G$ ($y_1$ and $y_2$ are both adjacent to $x_2$).   
 This is a contradiction. 

\medskip
\noindent
\textit{Case 2:} The two edges in $N(x)$ share no endpoints.  
\smallskip

Assume the edges in $N(x)$ are $x_1 x_2$ and $x_3 x_4$.    
Then $V_{12}$ and $V_{34}$ must be empty (see \cite{aehk}).  
Since $N(y) \cap N(x)$ must contain an edge for each $y \in Y$ and 
$V_{12} = V_{34} = \emptyset$,
we can partition $Y$ as follows:
\[
Y = V_{123} \cup V_{124} \cup V_{134} \cup V_{234} \cup V_{1234}.
\]
Note that every vertex in $Y$ is adjacent to either both $x_1$ and $x_2$, or to both $x_3$ and $x_4$.

Suppose $y \in V_{1234}$.  We claim that the degree of $y$ is exactly 4.  
If $y$ is adjacent to some
$z \in Y \backslash \{y \}$, then we may assume, without loss of generality,
that $z$ is adjacent to both $x_1$ and $x_2$.  However, this implies 
$\{y,z,x_1,x_2 \}$ is a $K_4$ in $G$ which is a contradiction.
Consequently, $d(y) = 4$ whenever $y \in V_{1234}$, and $y$ lies in
exactly two triangles: $yx_1x_2$ and $yx_3 x_4$.  

We now focus on 
\[
Y' :=  V_{123} \cup V_{124} \cup V_{134} \cup V_{234}.
\]  
The intersection $N(x_1) \cap N(x_3)$ 
must contain an edge $\alpha \beta$.  This edge 
has both endpoints in $Y$
since the only common neighbor of $x_1$ and $x_3$ in $N[x]$ 
is $x$.  
Because $\alpha$ and $\beta$ are adjacent vertices in $Y$, 
neither can be in $V_{1234}$.  Thus, 
$\alpha ,\beta \in V_{123} \cup V_{134}$.  
If $\alpha $ and $\beta$ are both in $V_{123}$, then $\{\alpha , \beta , x_1 ,x_2 \}$ is a $K_4$.
A similar argument shows $\alpha$ and $\beta$ cannot both 
be in $V_{134}$.  We may assume that $\alpha \in V_{123}$
and $\beta \in V_{134}$, which shows $V_{123} \neq \emptyset$ and $V_{134} \neq \emptyset$.
By symmetry, $V_{124}$ and $V_{234}$ are not empty.  

It is proved in \cite{aehk} that $G[Y']$ is a complete bipartite graph with parts
\begin{center}
$V_{123} \cup V_{124}$ and $V_{134} \cup V_{234}$.
\end{center}
This determines the graph $G$ up to the number of vertices in the parts
$V_{123}$, $V_{124}$, $V_{134}$, and $V_{234}$ (which are all not empty),
and $V_{1234}$ (which may be empty).  
 
Let $y_{123} \in V_{123}$, and define $y_{124}$, $y_{134}$, and $y_{234}$ similarly.
The subgraph $G_{1}'$ of $G$ induced by $N[x] \cup \{  y_{123} , y_{124}, y_{134}, y_{234} \}$ has
9 vertices and 14 triangles; it is exactly one of the graphs appearing in the proof of Theorem 8 in \cite{aehk}.  These triangles are
\begin{center}
$x x_1 x_2$, $x x_3 x_4$, $x_1 x_2 y_{123}$, $x_1 x_2 y_{124}$, $x_3 x_4 y_{134}$, $x_3 x_4 y_{234}$
$y_{123} y_{134} x_1$, $y_{123} y_{134} x_3$, 

$y_{134} y_{124} x_1$, $y_{134} y_{124} x_4$,
$y_{124} y_{234} x_2$, $y_{124} y_{234} x_4$,
$y_{123} y_{234} x_2$, $y_{123} y_{234} x_3$.
\end{center}  

We now estimate the number of triangles in $G$ by adding back the remaining $n-9$ vertices in $V(G) \backslash V(G_1')$ 
one by one.  At each step, we count the number of triangles that contain the new added vertex
and vertices in $V(G_1')$.  
By counting triangles in this way, we never count the same triangle more than once.  

If a vertex $y$ is added to $V_{1234}$, then we obtain exactly two new triangles: $y x_1 x_2$ and 
$y x_3 x_4$.  We can add any number of vertices to $V_{1234}$ and the resulting graph 
is $K_4$-saturated.  
Now suppose a vertex $y$ is added to $V_{123}$ (the other cases are the same by symmetry).  
This creates at least five new triangles: $y x_1 x_2$, $y y_{134} x_1$, $y y_{134} x_3$,
$y y_{234} x_2$, and $y y_{234} x_3$.  
The conclusion is that 
\[
k_3 (G) \geq k_3 (G_1') + 2(n-9) = 14 + 2(n-9) = 2n-4.
\]
We have $k_3(G) = 2n-4$ 
if and only if all of the $n-9$ vertices $V(G) \backslash V(G_1')$ are contained in $V_{1234}$. 
This is precisely the graph $H_4(n)$.
\end{proof}

\bigskip

\noindent
\textbf{Method and Notation For Triangle Counting:} The counting method used in the last three paragraphs of the proof of Lemma 
\ref{lemma 2 edges} will be used multiple times in the proof of Lemmas \ref{claim 3 edges} and \ref{lemma 4 edges}.
We find a small subgraph $G_i'$ of $G$, count the number of triangles in $G_i'$, and then count 
the number of triangles created when a vertex $y$ is added to $G_i'$.  The crucial point is that 
when a vertex $y$ is added, we are only counting triangles that contain $y$ and vertices in $G_i'$.  
This means that we never count the same triangle more than once.  
In an effort to be concise, we will always write $y$ for the added vertex, and then list the triangles in the following way.
If $y$ is added to $V_S$ and this creates triangles $y \alpha_1 \beta_1$, $y \alpha_2 \beta_2, \dots , y \alpha_k \beta_k$,
we will write
\[
V_S ~ : ~ y \alpha_1 \beta_1, y \alpha_2 \beta_2, \dots , y \alpha_k \beta_k.
\]
We preface this with the statement ``When $y$ is added to $V_S$,'' and then list the 
triangles containing $y$ using the notation shown above.  

\bigskip

\begin{lemma}\label{claim 3 edges}
Suppose $n \geq 12$. 
If the number of edges in $N(x)$ is 3, then 
\[
k_3 (G) \geq 3n -18.
\]
\end{lemma}
\begin{proof}[Proof of Lemma \ref{claim 3 edges}]
Since $N(x)$ must be triangle free, the three edges in $N(x)$ either form a star or a path of length three.

Suppose first the edges in $N(x)$ are $x_1 x_2$, $x_1 x_3$, and $x_1 x_4$.  
Since every edge in $N(x)$ has $x_1$ as an endpoint and every vertex $y \in Y$ must 
be joined to an edge in $N(x)$, vertex $x_1$ is adjacent to all vertices in $G$. 
Let $H$ be the subgraph of $G$ induced by $N(x_1)$.
Then $H$ is a $K_3$-saturated $(n-1)$-vertex graph with $\delta (H) = 3$.  
Since $n \geq 11$, 
a result of Duffus and Hanson \cite{dh} implies that $H$ has 
at least $3(n-1) -15 = 3n - 18$ edges.  Each of these edges forms a triangle with $x$ and 
so $k_3(G) \geq 3n -18$.  

Assume now the edges in $N(x)$ are $x_1 x_2$, $x_2 x_3$, and $x_3 x_4$.  
First we show that $V_{12}$ and $V_{34}$ are empty.  
We prove this in full, but then 
we will make use of the results proved in \cite{aehk} concerning adjacencies 
between a $V_S$ and a $V_T$. 
If 
$y \in V_{12}$, then $y$ is not adjacent to $x_4$, and so there 
must be an edge $\alpha \beta $ in the intersection $N(y) \cap N (x_4)$.  
Since $y$ and $x_4$ have no common neighbors in $N[x]$,
the endpoints of $\alpha \beta$ lie in $Y$.  If $\alpha$ or $\beta$ is 
adjacent to both $x_1$ and $x_2$,
then either $\{x_1 , x_2 , y , \alpha \}$ or 
$\{ x_1 , x_2 , y , \beta \}$ is a $K_4$.    
The sets $N( \alpha ) \cap N( x)$ and $N( \beta ) \cap N(x)$ 
must contain at least one of the edges in $N(x)$, other than $x_1 x_2$.
All of the edges in $N(x)$, other than $x_1 x_2$, have $x_3$ as an endpoint
and so $\alpha$ and $\beta$ are both adjacent to $x_3$, but then
$\{ \alpha , \beta , x_3 , x_4 \}$ is a $K_4$.  The conclusion is that 
$V_{12} = \emptyset$.  By symmetry, $V_{34} = \emptyset$.  

As in the proof of Lemma \ref{lemma 2 edges}, 
if $y \in V_{1234}$, then $d(y) = 4$ and $N(y) = \{x_1 ,x_2 ,x_3 ,x_4 \}$.  
Let $Y' = Y  \backslash V_{1234}$.  Because $V_{12} = V_{34} = \emptyset$, we have
\begin{equation}\label{3 edge 1}
Y' := V_{23} \cup V_{123} \cup V_{124} \cup V_{134} \cup V_{234}.
\end{equation}

If $y$ and $z$ are 
adjacent vertices with $y,z \in  V_{23} \cup V_{123} \cup V_{234}$,
then $\{ y , z , x_2 , x_3 \}$ is a $K_4$.  Thus,
\begin{equation}\label{3 edge 2}
V_{23} \nsim V_{123} ~~~~~ \mbox{and} ~~~~~ V_{23} \nsim V_{234} 
~~~~~\mbox{and}~~~~~ V_{123} \nsim V_{234}.
\end{equation}
Similar arguments show that 
\[
V_{123} \nsim V_{124}, V_{134} \nsim V_{234}, V_{23} \sim V_{124}, V_{23} \sim V_{134}, 
V_{123} \sim V_{134}, V_{124} \sim V_{134} ~\mbox{and}~ V_{124} \sim V_{234}.
\]
Summarizing, $V_S \sim V_T$ if the intersection $S \cap T$ does not 
contain one of the pairs $\{1,2 \}$, $\{2,3 \}$, or $\{3,4 \}$.  Also,
$V_S \nsim V_T$ if $S \cap T$ contains at least 
one of the pairs $\{ 1,2 \}$, $\{2,3 \}$, or $\{3,4 \}$.  
We represent this using a graph.
A solid edge between
$V_S$ and $V_T$ indicates $V_S \sim V_T$, a dashed edge indicates $V_S \nsim V_T$,
and no edge between $V_S$ and $V_T$ indicates that it is possible for a vertex in $V_S$ to be 
adjacent or not adjacent to vertex in $V_T$ (this last possibility does not occur 
here, but will occur in the proof of Lemma \ref{lemma 4 edges}).  

\begin{center}

\begin{tikzpicture}[scale=.9]

\draw[dashed] (0,2) -- (1,0);
\draw (0,2) -- (3,0);
\draw[dashed]  (0,2) -- (4,2);
\draw[dashed] (0,2) -- (2,3);

\draw (1,0) -- (3,0);
\draw (1,0) -- (4,2);
\draw (1,0) -- (2,3);

\draw (2,3) -- (3,0);
\draw[dashed] (2,3) -- (4,2);

\draw[dashed] (3,0) -- (4,2);

\draw[fill=black] (0,2) circle (0.07cm);
\draw[fill=black] (1,0) circle (0.07cm);
\draw[fill=black] (2,3) circle (0.07cm);
\draw[fill=black] (3,0) circle (0.07cm);
\draw[fill=black] (4,2) circle (0.07cm);

\draw (-.1,2) node[left]{$V_{123}$};
\draw (1,-.1) node[below]{$V_{124}$};
\draw (3,-.1) node[below]{$V_{134}$};
\draw (4.1,2) node[right]{$V_{234}$};
\draw (2,3.1) node[above]{$V_{23}$};

\end{tikzpicture}

Adjacencies among the $V_S$ and $V_T$
\end{center}

From Lemma \ref{general lemma}, we have that two Rules for this graph are 
\begin{center}

\begin{tikzpicture}[scale=.6]

\draw[->] (-.1,1.9) -- (2.7,0);

\draw (-.1,2) node[left]{$V_{123}$};
\draw (.5,-.1) node[below]{$V_{124}$};
\draw (3.5,-.1) node[below]{$V_{134}$};
\draw (4.1,2) node[right]{$V_{234}$};
\draw (2,3.1) node[above]{$V_{23}$};


\draw[->] (1+9,0) -- (2.9+9,0);
\draw[->] (1+9,0) -- (4+9,2);
\draw[->] (1+9,0) -- (2+9,3);

\draw (-.1+9,2) node[left]{$V_{123}$};
\draw (1+8.5,-.1) node[below]{$V_{124}$};
\draw (3+9.6,-.1) node[below]{$V_{134}$};
\draw (4.1+9,2) node[right]{$V_{234}$};
\draw (2+9,3.1) node[above]{$V_{23}$};

\end{tikzpicture}

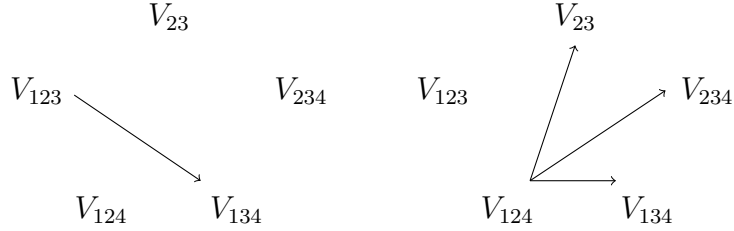
Rule 1 (on the left) and Rule 2 (on the right)
\end{center}
By symmetry (which is highlighted by the way we have chosen to show Rules 1 and 2), 
we have $V_{234} \rightarrow V_{124}$ and $V_{134} \rightarrow V_{23} \cup V_{123} \cup V_{124}$.  
We also refer to the assertion $V_{234} \rightarrow V_{124}$ as Rule 1, and 
the assertion $V_{134} \rightarrow V_{23} \cup V_{123} \cup V_{124}$ as Rule 2.  

We take a moment to carefully show how Lemma \ref{general lemma} gives Rule 2.  
Choosing $S = \{1,2,4 \}$ and 
$i = 3$, we find all $T$ for which $ 3 \in T$, and $|T \cap \{1,2 \} | \leq 1$.   
These two conditions fail for $\{1,2,3 \}$, but hold for $\{2,3 \}$, $\{1,3,4 \}$, and $\{2,3,4 \}$.  
Thus, $V_{124} \rightarrow V_{23} \cup V_{134} \cup V_{234}$.  

The structure of $G$ is now determined up to the sizes of the parts
$V_{1234}$, $V_{23}$, $V_{123}$, $V_{124}$, $V_{134}$, and $V_{234}$.

If all of the vertices in $Y$ belong to $V_{1234}$, then 
$k_3 (G) \geq 3 + 3( n- 5) = 3n - 12$
and we are done.  Let $Y' := Y \backslash V_{1234}$ and assume $Y' \neq \emptyset$.

\medskip
\noindent
\textit{Case 1:} $V_{23} \neq \emptyset$.

\medskip
Let $y_{23} \in V_{23}$.  Since $\delta (G) \geq 4$, $y_{23}$ has at least two other neighbors
aside from $x_2$ and $x_3$.  Let $z_1$, $z_2$ be two other neighbors of $y_{23}$.  These two neighbors
must be in the union $V_{124} \cup V_{134}$.  

First suppose $z_1:= z_{124}  \in V_{124}$ and $z_2 : = z_{134} \in V_{134}$.    
Let $G_2'$ be the subgraph of $G$ induced by $N[x] \cup \{y_{23},z_{124} ,z_{134} \}$.  This subgraph
has 8 vertices and 11 triangles (see the Appendix for the list of triangles).  
When $y$ is added to $V_S$, 
\begin{center}
$V_{23} ~ : ~ y x_2 x_3, y z_{124} x_2, y z_{134} x_3$, $y z_{124} z_{134}$ 

\smallskip

$V_{123} ~ : ~ y x_1 x_2, y x_2 x_3 , y x_1 z_{134} ,y x_3 z_{134}$ ($V_{234}$ is the same by symmetry)

\smallskip

$V_{124} ~ : ~ y x_1 x_2, y x_2 y_{23}, y z_{134} y_{23}$ ($V_{134}$ is the same by symmetry)
\end{center}
The conclusion is that 
\[
k_3 (G) \geq k_3 (G_2') + 3(n-8) = 11 + 3 ( n - 8) = 3n - 13.
\]  

Now suppose that $z_1:=z_{124}$ and $z_2:=z_{124}'$ are in $V_{124}$, and 
that $V_{134} = \emptyset$.  
If $V_{123} \neq \emptyset$, then by Rule 2, $V_{123} \rightarrow V_{134}$ implying $V_{134} \neq \emptyset$,
a contradiction.  Therefore, $V_{123} = \emptyset$.  

We consider two subcases.
\begin{center}
\textit{Subcase 1:} $V_{234} \neq \emptyset$.
\end{center}
Let $z_{234} \in V_{234}$ and $G_3'$ be the subgraph of $G$ induced by
$N[x] \cup \{ y_{23} , z_{124} , z_{124}' , z_{234} \}$.  
This subgraph has 9 vertices and 14 triangles.  
When $y$ is added to $V_S$,
\begin{center}
$V_{23} ~ : ~ y x_2 x_3, y z_{124} x_2, y z_{124}' x_2$ 
~~~~~
$V_{124} ~ : ~ y x_1 x_2, y y_{23} x_2, y z_{234} x_2, y z_{234} x_4$  

$V_{234} ~ : ~ y x_2 x_3, y x_3 x_4, y z_{124} x_2, y z_{124}' x_2$
\end{center}
We conclude that $k_3(G) \geq k_3(G_3') + 3 (n-9) =  14 + 3( n- 9) =  3n-13$.    

\medskip
\begin{center}
\textit{Subcase 2:} $V_{234} = \emptyset$.
\end{center}

The neighborhood $x_1 x_4$ must contain an edge, say $\alpha \beta$.  This edge cannot 
use the vertex $x$ since $x_1$ is not adjacent to $x_3$, and $x_4$ is not adjacent to $x_2$.
Furthermore, neither endpoint is in $V_{23}$ because $x_1$ is not adjacent to any vertex in
 $V_{23}$.  We conclude that the endpoints of $\alpha$ and $\beta$ must both be in
 $V_{124}$, but then $\{ \alpha , \beta , x_1 ,x_2 \}$ is a $K_4$.

\bigskip
\noindent
\textit{Case 2:} $V_{23} = \emptyset$.

\medskip

Here we will consider two subcases.
\begin{center}
\textit{Subcase 1:} $V_{123} \neq \emptyset$
\end{center}
Let $y_{123} \in V_{123}$.  By Rule 1, 
$V_{123} \rightarrow V_{134}$ so there is a $y_{134} \in V_{134}$,
and $y_{123}$ is adjacent to $y_{134}$.

If $V_{234} \neq \emptyset$, say $y_{234} \in V_{234}$,
then by Rule 1, $V_{234} \rightarrow V_{124}$.   Let $ y_{124} \in V_{124}$ 
be a neighbor of $y_{234}$.  
Let
$G_4'$ be the subgraph of $G$ induced by
$N[x] \cup \{  y_{123} , y_{134} ,y_{234}, y_{124} \}$.
This graph has 9 vertices and 15 triangles.  When $y$ is added to $V_S$, 
\begin{center}
$V_{123} ~:~ y x_1 x_2, y x_2 x_3, y y_{134} x_1, y y_{134} x_3$ ($V_{234}$ is the same by symmetry)

$V_{134} ~:~ y x_3 x_4, y y_{123} x_1, y y_{123} x_3, y y_{124} x_4$ ($V_{124}$ is the same by symmetry)
\end{center}
Thus,
\[
k_3 (G) \geq k_3 (G_4' ) + 4 (n-9) =  15+ 4(n-9)  = 4n - 21 \geq 3n-14.
\] 

Now suppose $V_{234} = \emptyset$.   
Let $G_5' $ be the subgraph induced by $N[x] \cup \{ y_{123} ,y_{134} \} $.  The graph $G_5'$ has 7 vertices and 8 triangles.
When $y$ is added to $V_S$,
\begin{center}
$V_{123} ~:~ y x_1 x_2, y x_2 x_3, y y_{134} x_1, y y_{134} x_3$
~~~~~
$V_{134} ~:~ y  x_3 x_4, y y_{123} x_1, y y_{123} x_3$

\smallskip

$V_{124} ~:~ y x_1 x_2, y y_{134} x_1, yy_{134} x_4$
\end{center}  
Thus, $k_3 (G) \geq k_3 ( G') + 3(n-7) = 8 + 3 ( n- 7 )  = 3n - 13$.

\medskip
\begin{center}
\textit{Subcase 2:} $V_{123} = \emptyset$. 
\end{center}

By symmetry, we may assume $V_{234} = \emptyset$ (otherwise we are 
back in Subcase 1 with $V_{234}$ replacing $V_{123}$).  Then all of the vertices of $G$ not in 
$N[x]$ must be in $V_{124} \cup V_{134} \cup V_{1234}$.   Let $y_{124} \in V_{124}$.
By Rule 1, $V_{124} \rightarrow V_{23} \cup V_{234} \cup V_{134}$, but
$V_{23} = V_{234} = \emptyset$.  Thus, there is a $y_{134} \in V_{134}$.  If we take 
$G_6' $ to be the subgraph of $G$ induced by $N[x] \cup \{ y_{124} ,y_{134} \}$, then 
$G_6'$ has 6 vertices and 7 triangles.  
When $y$ is added to $V_S$,
\begin{center}
$V_{124} ~ :~ y x_1 x_2, y y_{134} x_1 , y y_{134} x_4$ ($V_{134}$ is the same by symmetry)
\end{center}
Therefore, $k_3 (G) \geq 3 (n - 7 ) + 7 = 3n -14$.
\end{proof}

\begin{lemma}\label{lemma 4 edges}
If $n \geq 15$ and the number of edges in $N(x)$ is 4, then 
\[
k_3 (G) \geq 2n-3.
\]
\end{lemma}
\begin{proof}[Proof of Lemma \ref{lemma 4 edges}]
Since $G$ is $K_4$-free, $N(x)$ must be triangle-free and so the four edges in $N(x)$ form
a $C_4$.  Assume $x_1 x_2$, $x_2 x_3$, $x_3 x_4$, and $x_4 x_1$ are the four 
edges in $N(x)$.  
The set 
$Y$ can be partitioned into the disjoint union
\[
Y = V_{12} \cup V_{23} \cup V_{34} \cup V_{14} \cup V_{123} \cup V_{124} \cup V_{134} \cup V_{234} \cup V_{1234}.
\]
The first step is to deal with vertices in $V_{1234}$.  If $y \in V_{1234}$, then 
$y$ cannot have a neighbor in $Y$, otherwise we obtain a $K_4$.  Thus, $N(y) = \{ x_1 ,x_2 ,x_3 ,x_4 \}$
for all $y \in V_{1234}$.  Such a vertex lies in four triangles: $y x_1x_2$, $y x_2 x_3$, $y x_3 x_4$,
and $y x_4 x_1$.  
We will therefore assume that $V_{1234} = \emptyset$ (if $G'$ is the subgraph of $G$ obtained by
removing the vertices in $Y_{1234}$ and we can prove the 
result for $G'$, then the result easily follows for $G$).  

For the rest of the proof, we focus on 
\[
Y' := V_{12} \cup V_{23} \cup V_{34} \cup V_{14} \cup V_{123} \cup V_{124} \cup V_{134} \cup V_{234}.
\]

As stated in \cite{aehk}, the following relationships hold among these 8 sets.  Recall a solid/dashed edge
indicates that all vertices of $V_S$ are adjacent/not adjacent to all vertices of $V_T$.

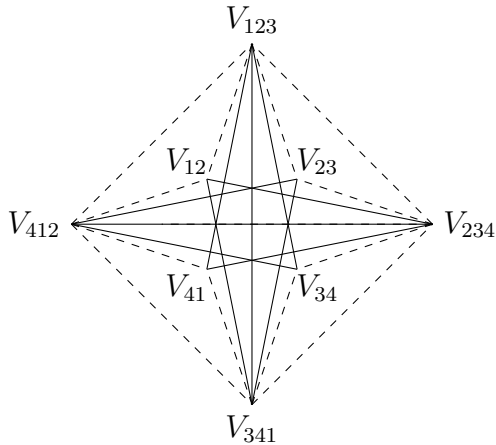
\begin{center}

\begin{tikzpicture}[scale=.6]

\draw[dashed] (0,4) -- (4,0);
\draw[dashed] (0,4) -- (3,3);
\draw (0,4) -- (5,3);
\draw (0,4) -- (8,4);
\draw (0,4) --(5,5);
\draw[dashed] (0,4) -- (3,5);
\draw[dashed] (0,4) -- (4,8);

\draw[dashed] (4,0) -- (3,3);
\draw (4,0) -- (3,5);
\draw (4,0) -- (4,8);
\draw (4,0) -- (5,5);
\draw[dashed] (4,0) -- (5,3);
\draw[dashed] (4,0) -- (8,4);

\draw[dashed] (8,4) -- (5,3);
\draw[dashed] (8,4) -- (2,4);
\draw (8,4) -- (3,3);
\draw (8,4) -- (3,5);
\draw[dashed] (8,4) -- (4,8);
\draw[dashed] (8,4) -- (5,5);

\draw[dashed] (4,8) -- (3,5);
\draw (4,8) -- (3,3);
 \draw (4,8) -- (5,3);
\draw[dashed] (4,8) -- (5,5);

\draw (0,4) node[left]{$V_{412}$};
\draw (4,8) node[above]{$V_{123}$};
\draw (8,4) node[right]{$V_{234}$};
\draw (4,0) node[below]{$V_{341}$};

\draw (2.55,2.6) node{$V_{41}$};
\draw (2.55,5.4) node{$V_{12}$};
\draw (5.45,5.4) node{$V_{23}$};
\draw (5.45,2.6) node{$V_{34}$};

\end{tikzpicture}

\end{center}
The only missing edges in this figure are those with endpoints in $\{ V_{12}, V_{23}, V_{34},V_{41} \}$. At this 
stage, we do not have enough information to determine adjacencies between these sets.  

Next, we apply Lemma \ref{general lemma} to obtain Rules 3 and 4. 
\begin{center}

\begin{tikzpicture}[scale=1]

\draw [->] (1,3)--(2.2,3);
\draw [->] (1,3)--(3.1,2);

\draw [->] (.65,2.75)--(.65,1.3);
\draw [->] (.65,2.75)--(1.65,.3);

\draw[fill=black] (0,2) node[left]{$V_{412}$};
\draw[fill=black] (1,1) node[left]{$V_{41}$};
\draw[fill=black] (1,3) node[left]{$V_{12}$};
\draw[fill=black] (2,0) node[left]{$V_{341}$};
\draw[fill=black] (2,4) node[left]{$V_{123}$};
\draw[fill=black] (3,1) node[left]{$V_{34}$};
\draw[fill=black] (3,3) node[left]{$V_{23}$};
\draw[fill=black] (4,2) node[left]{$V_{234}$};

\draw [->] (1.6+6,3.7)--(.8+6,1.3);
\draw [->] (1.6+6,3.7)--(1.6+6,.3);
\draw [->] (1.6+6,3.7)--(2.5+6,1.3);

\draw[fill=black] (0+6,2) node[left]{$V_{412}$};
\draw[fill=black] (1+6,1) node[left]{$V_{41}$};
\draw[fill=black] (1+6,3) node[left]{$V_{12}$};
\draw[fill=black] (2+6,0) node[left]{$V_{341}$};
\draw[fill=black] (2+6,4) node[left]{$V_{123}$};
\draw[fill=black] (3+6,1) node[left]{$V_{34}$};
\draw[fill=black] (3+6,3) node[left]{$V_{23}$};
\draw[fill=black] (4+6,2) node[left]{$V_{234}$};

\end{tikzpicture}

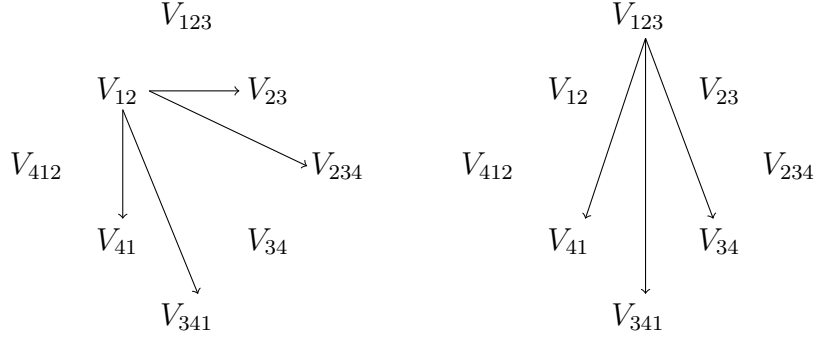
Rule 3 (on the left) and Rule 4 (on the right)
\end{center}
Again, the pictures are drawn to highlight the symmetry, and we use Rule 3 and Rule 4
accordingly (for example, the assertion $V_{412} \rightarrow V_{23} \cup V_{234} \cup V_{34}$ is also
called Rule 4).  Rule 3 is obtained by applying Lemma \ref{general lemma} with $S = \{1,2 \}$ and 
$i = 3$, which gives $V_{12} \rightarrow V_{23} \cup V_{234}$, and then
again with $S = \{ 1 ,2 \}$ with $i =4$, which gives $V_{12} \rightarrow V_{41} \cup V_{341}$.  

The proof of Lemma \ref{lemma 4 edges} from this point forward will be divided into cases.

\medskip
\noindent
\textit{Case 1:} $V_{12} \cup V_{23} \cup V_{34} \cup V_{14} = \emptyset$.

\medskip

Since $G$ has more than 5 vertices, we can assume there is some $y_{123} \in V_{123}$.  
By Rule 4, $V_{123} \rightarrow V_{41} \cup V_{341} \cup V_{34}$, 
but $V_{41} = V_{34} = \emptyset$, so $y_{123}$ has a neighbor $y_{134} \in V_{134}$.
Hence, we know that $V_{123} \neq \emptyset$ and $V_{134} \neq \emptyset$.  

 \medskip
\begin{center}
\textit{Subcase 1:}  $V_{234} \cup V_{124}= \emptyset$. 
\end{center}

Let $G_7'$ be the subgraph of $G$ induced by $N[x] \cup \{ y_{123} , y_{134} \}$.  Then $G_7'$ has 
7 vertices and 10 triangles.  When $y$ is added to $V_S$,   
\begin{center}
$V_{123} ~ : ~ y x_1 x_2 , y x_2 x_3 , y y_{134} x_1 , y y_{134} x_3$ ~~ ($V_{341}$ is the same by symmetry)
\end{center}
Thus, $k_3 (G) \geq 10 +  4 (n-7) = 4n -18 \geq 2n-2$.

 \medskip
\begin{center}
\textit{Subcase 2:} All of $V_{123}$, $V_{134}$, $V_{234}$, and $V_{124}$ are not empty.
\end{center}

Before dealing with Subcase 2, we note that Subcases 1 and 2 do cover all possibilities.  
This is because if $V_{234} \neq \emptyset$ (or $V_{124} \neq \emptyset$), then Rule 4 and 
the fact that $V_{12} \cup V_{23} \cup V_{34} \cup V_{41} = \emptyset$ implies 
$V_{124} \neq \emptyset$ (or $V_{234} \neq \emptyset$).  Therefore, 
it cannot be the case that exactly one of $V_{123}$, $V_{134}$, $V_{234}$, $V_{124}$ is empty.

Assume $y_{234} \in V_{234}$ and $y_{412} \in V_{412}$.  
Let $G_8'$ be the subgraph of $G$ induced by 
$ N[x] \cup \{  y_{123} , y_{234} , y_{134} , y_{412} \}$. 
This graph has 9 vertices and 16 triangles.  When $y$ is added to $V_S$,
\begin{center}
$V_{123} ~ : ~ y x_1 x_2 , y x_2 x_3 , y y_{134} x_1 , y y_{134} x_3$ ~~ ($V_{234}$, $V_{134}$, $V_{412}$ are the
same)
\end{center}
Thus, $k_3 (G) \geq 16 + 4( n-9) = 4n -20 \geq 2n-2$.

\medskip
\noindent
\textit{Case 2:} $V_{12} \cup V_{23} \cup V_{34} \cup V_{14} \neq \emptyset$.

\medskip 
 For Case 2 we need an additional rule which is proved in \cite{aehk}.

\bigskip
\noindent
\textbf{Rule 5:} If $F$ is the subgraph of $G$ induced by $V_{12} \cup V_{23} \cup V_{34} \cup V_{14}$,
then $F$ is a 4-partite graph that is $K_4$-saturated with respect to the parts.

\bigskip

In particular, Rule 5 implies that if one of the four parts $V_{12},V_{23},V_{34},V_{14}$ is empty, 
then the remaining parts induce a complete
1-partite, 2-partite, or 3-partite graph depending on the number of nonempty parts.

\begin{center}
\textit{Subcase 1:} $V_{12} \neq \emptyset$ and $V_{23} = V_{34} = V_{14} = \emptyset$.  
\end{center}

Let $y_{12} \in V_{12}$.  By Rule 3,
$V_{12} \rightarrow V_{41} \cup V_{341}$ and 
$V_{12} \rightarrow V_{23} \cup V_{234}$, but
 $V_{41} = V_{23} = \emptyset$ so
there must be vertices $y_{234} \in V_{234}$ and 
$y_{134} \in V_{134}$ that are both adjacent to $y_{12}$.
Let $G_9'$ be the subgraph of $G$ induced 
by $N[x] \cup \{  y_{12} , y_{234} , y_{134} \}$.  
Then $G'$ has 11 triangles and 8 vertices.  
When $y$ is added to $V_S$,
\begin{center}
$V_{12} ~ :  ~  y x_1 x_2, y y_{234}x_2,  y y_{134} x_1$
~~~~~~~~~~
$V_{123} ~ : ~ y x_1 x_2, y x_2 x_3, y y_{134} x_1$ ($V_{412}$ is the same)

\smallskip

$V_{234} ~ : ~ y x_2 x_3 , y x_3 x_4 , y y_{12} x_2$ ($V_{341}$ is the same)

\end{center}
We conclude that 
 \[
 k_3 (G) \geq k_3 (G_9') + 3 ( n - 8) = 11 + 3(n-8) = 3n - 13 \geq 2n-2
 \]
 where we have used the assumption $n \geq 15$ for the last inequality.  
 
\begin{center}
\textit{Subcase 2:} $V_{12} \neq \emptyset$ and $V_{23} \neq \emptyset$, $V_{34} = V_{14} = \emptyset$.  
\end{center}

Let $y_{12} \in V_{12}$ and $y_{23} \in V_{23}$. 
By Rule 5, any vertex in $V_{12}$ will be adjacent to all vertices in $V_{23}$.  
By Rule 3, $V_{12} \rightarrow V_{41} \cup V_{341}$ so, since $V_{41} = \emptyset$, 
there is a vertex $y_{134} \in V_{134}$.  This vertex must be adjacent to both $y_{12}$ and $y_{23}$.  
Let $G_{10}'$ be the subgraph of $G$ induced by $N[x] \cup \{  y_{12},y_{23} , y_{134} \}$.
Then $G_{10}'$ has 8 vertices and 12 triangles.  When $y$ is added to $V_S$,
 \begin{center}
 $V_{12}$~:~ $y x_1 x_2$ $y y_{23} x_2$, $y y_{134} x_1$ ($V_{23}$ is the same),
 
 \smallskip
 
 $V_{123}$: $y x_1 x_2$, $y x_2 x_3$, $y y_{134} x_1$,$y y_{134} x_4$~~~~~~~~~~
  $V_{134}$ :$yx_1 x_4$, $yx_3 x_4$, $y y_{12} x_1$,
 
 \smallskip
 
 $V_{234 }$: $y x_2 x_3$, $yx_3 x_4$, $y y_{12} x_2$ ($V_{412}$ is the same).
 \end{center}
 Therefore,
 $k_3(G) \geq 12 + 3(n-8) = 3n - 12 \geq 2n-2$.

\begin{center}
\textit{Subcase 3:} $V_{12} \neq \emptyset$ and $V_{34} \neq \emptyset$, $V_{23} = V_{14} = \emptyset$.  
\end{center}

Let $y_{12} \in V_{12}$ and $y_{34} \in V_{34}$.  By Rule 3,  
$V_{12} \rightarrow V_{23} \cup V_{234}$, $V_{12} \rightarrow V_{41} \cup V_{341}$,
but $V_{23} = V_{14} = \emptyset$.  Thus, 
there are $y_{134} \in V_{134}$ and $y_{234} \in V_{234}$ with 
both $y_{134}$ and $y_{234}$ adjacent to $y_{12}$.  
Also by Rule 3, $V_{34} \rightarrow V_{23} \cup V_{123}$ and
$V_{34} \rightarrow \cup V_{41} \cup V_{412}$.  This implies 
there are vertices $y_{123} \in V_{123}$ and $y_{412} \in V_{412}$
where both of these vertices are adjacent to $y_{34}$.  Let $G_{11}'$ be 
the subgraph of $G$ induced by 
$N[x] \cup \{ y_{12}, y_{34} , y_{134}, y_{234} , y_{123} , y_{412} \}$.  
Then $G'$ has 11 vertices and has 22 triangles.  When $y$ is added to $V_S$,
 \begin{center}
$V_{12}$: $y x_1 x_2$ $y y_{134} x_1$, $y y_{234} x_2$ ($V_{34}$ is the same),
 
 \smallskip

 $V_{123}$: $y x_1 x_2$, $y x_2 x_3$, $y y_{34} x_3$,
 $y y_{134} x_1$, $y y_{134} x_3$
  ($V_{234}$, $V_{134}$,
 and $V_{124}$ are the same).
 \end{center}
Thus, $k_3 (G) \geq k_3 (G_{11}') + 3 ( n - 11) = 3n - 11 \geq 2n-2$.

\begin{center}
\textit{Subcase 4:} $V_{12} \neq \emptyset$ , $V_{23} \neq \emptyset$, $V_{34} \neq \emptyset$, $V_{14} = \emptyset$.  
\end{center}

Let $y_{12} \in V_{12}$, $y_{23} \in V_{23}$, and $y_{34} \in V_{34}$.  
Note that $V_{12} \cup V_{23} \cup V_{34}$ is a complete 3-partite graph by Rule 5.
By Rule 3, $y_{34} \rightarrow V_{14} \cup V_{124}$, but $V_{14} = \emptyset$.
Therefore, there is a $y_{124} \in V_{124}$ and this vertex is adjacent to
both $y_{34}$ and $y_{23}$.  
Let $G_{12}'$ be the subgraph of $G$ induced by 
$N[x] \cup \{ y_{12}, y_{23}, y_{34} , y_{124} \}$.  
Then $G_{12}'$ has 9 vertices and has 15 triangles.  When $y$ is added to $V_S$,
 \begin{center}
$V_{12}$~:~ $y x_1 x_2$ $y y_{23} x_2$, $y y_{23} y_{34}$ 
~~~
$V_{23}$~:~ $y x_2 x_3$, $y y_{34} x_3$, $y y_{124} x_2$, $y y_{34} y_{12}$, 
$y y_{34} y_{124}$ 

\smallskip

$V_{34}$~:~ $y x_3 x_4$, $y y_{23} x_3$, $y y_{124} x_4$, $y y_{23} y_{12} ,y y_{23} y_{124}$  
~~~~~ 
$V_{123}$~:~$y x_1 x_2$ $y x_2 x_3$, $y y_{34} x_3$ 

\smallskip
$V_{234}$~:~ $y x_2 x_3$, $y x_3 x_4$, $y y_{12} y_2$,
~~~~~
$V_{124}$~:~$y x_1 x_2$, $y x_1 x_4$, $y y_{23} x_2$, $y y_{34} x_4$

\smallskip

$V_{341}$: $y x_3 x_4$, $y x_1 x_4$, $y y_{12} x_1$.
 \end{center}
Hence, $k_3 (G) \geq k_3 (G_{12}') + 3 ( n - 9) = 3n - 12$.

\begin{center}
\textit{Subcase 5:} Each of $V_{12}$, $V_{23}$, $V_{34}$, $V_{14}$ is not empty.    
\end{center}

For this subcase, we will count triangles in $G$ in a different way than the previous subcases.  
Let 
\begin{center}
$X = \{ x_1 ,x_2,x_3,x_4 \}$, $S = V_{12} \cup V_{23} \cup V_{34} \cup V_{41}$,
and 
$T = V_{123} \cup V_{234} \cup V_{341} \cup V_{412}$.
\end{center}

\begin{claim}\label{X claim}
The number of triangles that contain at least one vertex in $X$ is at least $2n-6$.
\end{claim}
\begin{proof}[Proof of Claim \ref{X claim}]
There are four triangles that contain $x$, 
$2 |T|$ triangles that contain two vertices in $X$ and one in $T$,
and $|S|$ triangles that contain two vertices in $X$ and one in $S$.  
By Rule 3, a vertex $y_{12} \in V_{12}$ lies in a triangle 
of the form $y_{12} z x_2$ where $z$ is some vertex in $V_{23} \cup V_{234}$.  Similar 
statements hold for vertices in $V_{23}$, $V_{34}$, and $V_{41}$.  Altogether,
we have $4 + 2 |T| + 2|S| = 4 +2 ( n - 5) = 2n - 6$ triangles.
\end{proof}

\bigskip

If $|S| = 4$, then we let $V_{ij} = \{ y_{ij} \}$ for $(i,j) \in \{ (1,2),(2,3),(3,4),(4,1) \}$. 
Since $G$ is $K_4$-free, Rule 5 implies that $S$ induces a $K_4 - e$.  
The two distinct cases, up to symmetry, are the missing edge $e$ is 
$y_{12} y_{23}$ or $y_{12} y_{34}$.

Suppose first $e = y_{12} y_{34}$.  For any $y_{123} \in V_{123}$, we have 
that $y_{123} y_{34} y_{14}$ is a triangle.  Likewise,
$y_{234} \in V_{234}$ implies $y_{234} y_{12} y_{14}$ is a triangle,
$y_{341} \in V_{341}$ implies $y_{341} y_{23} y_{12}$ is a triangle,
and
$y_{412} \in V_{234}$ implies $y_{412} y_{34} y_{12}$ is a triangle.
The two triangles $y_{12} y_{23} y_{34}$ and $y_{12} y_{41} y_{34}$ have no vertex in $X$.  The number
of triangles containing no vertex in $X$ is at least 
\[
2  + |T| = 2 + n -9 = n-7.
\]
Thus, by Claim \ref{X claim}, $G$ contains at least $3n -13 \geq 2n-2$ triangles.

Now suppose $e = y_{12} y_{23}$.  By Rule 3, $y_{12}$ has a neighbor 
in $V_{23} \cup V_{234} = \{ y_{23} \} \cup V_{234}$, but $y_{12}$ is not adjacent to 
$y_{23}$ so $V_{234} \neq \emptyset$.  Similarly, $y_{23}$ has a neighbor in $V_{12} \cup V_{412}$ 
but $y_{23}$ is not adjacent to $y_{12}$ so $V_{412} \neq \emptyset$.  
For any $y_{234} \in V_{234}$, $y_{234} y_{12} y_{41}$ is a triangle.
Likewise,
$y_{412} \in V_{412}$ implies $y_{412} y_{23} y_{34}$ is a triangle,
and
$y_{123} \in V_{123}$ implies $y_{123} y_{34} y_{41}$ is a triangle.
Therefore, there are at least 
\[
\gamma := 2 + |V_{234} | + |V_{412}| + |V_{123} | = 2 + |T| - |V_{314}| 
\]
triangles in $G$ with no vertex in $X$.  Since $V_{234}$ and $V_{412}$ are not empty,
\[
| V_{314} | \leq
 n - | X \cup \{x \} | - | S | - |V_{234} | - | V_{412} |
 \leq n - 5 - 4 - 1 -1 = n - 11.
\]
Hence, $\gamma \geq 2 + (n-9) - (n-11)  = 4$.  Combining this with Claim \ref{X claim} gives
\[
k_3 (G) \geq 2n-2.
\]

The final possibility is if at least one of the parts $V_{12}$, $V_{23}$, $V_{34}$, $V_{41}$ contains 
2 or more vertices.  Assume $|V_{12} | >1$ and let $y_{12}^1 , y_{12}^2 $ be distinct vertices in
$V_{12}$.  Let $y_{23} \in V_{23}$, $y_{34} \in V_{34}$, and $y_{41} \in V_{41}$.  
The set $\{ y_{12}^1 , y_{23} , y_{34} , y_{41} \}$ cannot form a $K_4$ and so by
Rule 5, this set of four vertices induces a $K_4-e$.  There are two triangles 
using these vertices.
Similarly, there are two triangles using the vertices 
$\{ y_{12}^2 , y_{23} , y_{34} , y_{41} \}$ and regardless of which
pair of vertices is not adjacent in this set, one of the triangles must contain
$y_{12}^2$.  Thus, we have 3 distinct triangles contained in $S$.  By Claim \ref{X claim},
\[
k_3 (G) \geq 2n - 3.
\]
This completes the proof of Lemma \ref{lemma 4 edges}.
\end{proof}

\bigskip

Combining Lemmas \ref{lemma 1 edge}, \ref{lemma 2 edges}, \ref{claim 3 edges}, 
and \ref{lemma 4 edges} implies Theorem \ref{main theorem}.


\section{Proof of Proposition \ref{lb 3}}\label{lb 3 section}

Let $s > 3$ and $t \geq 6 \binom{s-2}{2}$ be integers.  
Suppose $G$ is an $n$-vertex $K_s$-saturated graph
with minimum degree $t$ where $n \geq 2s -2$.  We must show that $G$ has at least 
$\binom{s-2}{2} (n -2)$ triangles. 

First assume every edge of $G$ lies in a triangle.  If $t(e)$ is the number of triangles that contain the edge $e$,
then the number of triangles in $G$ is 
\begin{eqnarray*}
\frac{1}{3} \sum_{e \in E(G) } t(e) \geq \frac{e(G)}{3} \geq \frac{tn}{6} \geq  \binom{s-2}{2} n.
\end{eqnarray*}

Now assume there is an edge $xy$ in $G$ that does not lie in any triangle.  Let $A = N(x)$ and $B = N(y)$.
Because $xy$ is not in any triangle, $A \cap B = \emptyset$.  
If $a \in A$, then $N(y) \cap N(a)$ must contain a copy of $K_{s-2}$.
The number of triangles that contain $a$ and an edge from this $K_{s-2}$ 
is $\binom{s-2}{2}$.  The same argument applies to a vertex $b \in B$.
Let $C = V(G) \backslash ( \{x,y \} \cup A \cup B )$.  
If $c \in C$, then both $N(c) \cap N(x)$ and $N(c) \cap N(y)$ must 
contain a copy of $K_{s-2}$.  These two copies of $K_{s-2}$ cannot 
share any edges because $A \cap B = \emptyset$.  Thus, $G$ must contain at least
\[
|A| \binom{s-2}{2} + |B| \binom{s-2}{2} + |C| \cdot 2 \binom{s-2}{2} 
\geq \binom{s-2}{2}(n-2)
\]
triangles.  This completes the proof of Proposition \ref{lb 3 section}.  


\section{Appendix}

For $G_1'$, the edges in $N(x)$ are $x_1 x_2$ and $x_3 x_4$.

\smallskip

14 Triangles in $G_1'$: 
$xx_1 x_2$, $x x_3 x_4$, 
$y_{123}x_1 x_2$, $y_{124} x_1 x_2$, $y_{134} x_3 x_4$, $y_{234} x_3 x_4$, 
$y_{123}y_{134} x_1$, $y_{123} y_{134}x_3$,
$y_{123} y_{234} x_2$, $y_{123} y_{234}x_3$, 
$y_{124} y_{134} x_1$, $y_{124} y_{134} x_4$, 
$y_{124} y_{234} x_2$,  $y_{124} y_{234} x_4$

\vspace{2em}

\noindent
For $G_2'$ through $G_6'$, the edges in $N(x)$ are $x_1 x_2$, $x_2 x_3$, and $x_3 x_4$.  

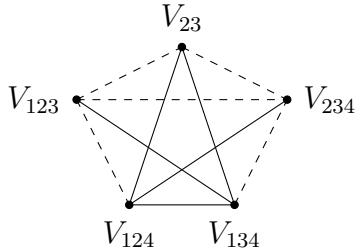
\begin{center}

\begin{tikzpicture}[scale=.7]

\draw[dashed] (0,2) -- (1,0);
\draw (0,2) -- (3,0);
\draw[dashed]  (0,2) -- (4,2);
\draw[dashed] (0,2) -- (2,3);

\draw (1,0) -- (3,0);
\draw (1,0) -- (4,2);
\draw (1,0) -- (2,3);

\draw (2,3) -- (3,0);
\draw[dashed] (2,3) -- (4,2);

\draw[dashed] (3,0) -- (4,2);

\draw[fill=black] (0,2) circle (0.07cm);
\draw[fill=black] (1,0) circle (0.07cm);
\draw[fill=black] (2,3) circle (0.07cm);
\draw[fill=black] (3,0) circle (0.07cm);
\draw[fill=black] (4,2) circle (0.07cm);

\draw (-.1,2) node[left]{$V_{123}$};
\draw (1,-.1) node[below]{$V_{124}$};
\draw (3,-.1) node[below]{$V_{134}$};
\draw (4.1,2) node[right]{$V_{234}$};
\draw (2,3.1) node[above]{$V_{23}$};

\end{tikzpicture}

\end{center}

11 Triangles in $G_2'$: $xx_1x_2$, $xx_2 x_3$, $x x_3 x_4$, $y_{23} x_2 x_3$, 
$z_{124} x_1 x_2$, $z_{134} x_3 x_4$,
$y_{23}z_{124}x_2$,
\\ $y_{23} z_{134}x_3$, $z_{124}z_{134}x_1$,
$z_{124}z_{134}x_4$,
$y_{23} z_{124} z_{134}$

\medskip

14 Triangles in $G_3'$: $xx_1x_2$, $xx_2 x_3$, $x x_3 x_4$, $y_{23} x_2 x_3$, 
$z_{124} x_1 x_2$, $z_{124}' x_1 x_2$,
$z_{234} x_2 x_3$, \\
$z_{234} x_3 x_4$,
$y_{23}z_{124}x_2$, $y_{23} z_{124}'x_2$, 
$z_{124}z_{234} x_2$, $z_{124} z_{234} x_4$, $z_{124}' z_{234} x_2$, $z_{124}' z_{234} x_4$
  
\medskip

15 Triangles in $G_4'$: $xx_1x_2$, $xx_2 x_3$, $x x_3 x_4$, 
$y_{123}x_1 x_2$, $y_{123} x_2 x_3$, $y_{124} x_1 x_2$,
$y_{234} x_2 x_3$, $y_{234} x_3 x_4$,
$y_{134} x_3 x_4$, 
$y_{123}x_1 y_{134}$, $y_{123} x_3 y_{134}$,
$y_{124} x_1 y_{134}$, $y_{124} x_4 y_{134}$,
$y_{124} x_2 y _{234} , y_{124} x_4 y_{234}$

\medskip

8 Triangles in $G_5'$: $xx_1x_2$, $xx_2 x_3$, $x x_3 x_4$, 
$y_{123}x_1 x_2$, $y_{123} x_2 x_3$, 
$y_{134} x_3 x_4$, 
$y_{123} y_{134} x_1$, $y_{123} y_{134} x_3$

\medskip

7 Triangles in $G_6'$: $xx_1x_2$, $xx_2 x_3$, $x x_3 x_4$, 
$y_{124} x_1 x_2$, $y_{134} x_3 x_4$, $y_{124} y_{134} x_1$, $y_{124} y_{134} x_4$

\vspace{2em}

\noindent
For $G_7'$ through $G_{12}'$, the edges in $N(x)$ are $x_1 x_2$, $x_2 x_3$, $x_3 x_4$, and $x_4 x_1$.  

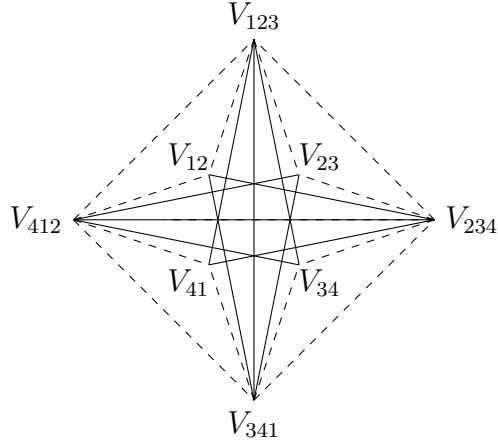
\begin{center}

\begin{tikzpicture}[scale=.6]

\draw[dashed] (0,4) -- (4,0);
\draw[dashed] (0,4) -- (3,3);
\draw (0,4) -- (5,3);
\draw (0,4) -- (8,4);
\draw (0,4) --(5,5);
\draw[dashed] (0,4) -- (3,5);
\draw[dashed] (0,4) -- (4,8);

\draw[dashed] (4,0) -- (3,3);
\draw (4,0) -- (3,5);
\draw (4,0) -- (4,8);
\draw (4,0) -- (5,5);
\draw[dashed] (4,0) -- (5,3);
\draw[dashed] (4,0) -- (8,4);

\draw[dashed] (8,4) -- (5,3);
\draw[dashed] (8,4) -- (2,4);
\draw (8,4) -- (3,3);
\draw (8,4) -- (3,5);
\draw[dashed] (8,4) -- (4,8);
\draw[dashed] (8,4) -- (5,5);

\draw[dashed] (4,8) -- (3,5);
\draw (4,8) -- (3,3);
 \draw (4,8) -- (5,3);
\draw[dashed] (4,8) -- (5,5);

\draw (0,4) node[left]{$V_{412}$};
\draw (4,8) node[above]{$V_{123}$};
\draw (8,4) node[right]{$V_{234}$};
\draw (4,0) node[below]{$V_{341}$};

\draw (2.55,2.6) node{$V_{41}$};
\draw (2.55,5.4) node{$V_{12}$};
\draw (5.45,5.4) node{$V_{23}$};
\draw (5.45,2.6) node{$V_{34}$};

\end{tikzpicture}

\end{center}

\medskip

10 Triangles in $G_7'$: $xx_1x_2$, $xx_2 x_3$, $x x_3 x_4$, $x x_4 x_1$, 
$y_{123} x_1 x_2$, $y_{123} x_2 x_3$, $y_{134} x_3 x_4$, $y_{134} x_4 x_1$,
$y_{123}y_{134} x_1$, $y_{123} y_{134} x_3$

\medskip

16 Triangles in $G_8'$: $xx_1x_2$, $xx_2 x_3$, $x x_3 x_4$, $x x_4 x_1$,
$y_{123} x_1 x_2$, $y_{123} x_2 x_3$, 
$y_{412} x_1 x_2$, $y_{412} x_4 x_1$, 
$y_{234} x_2 x_3$, $y_{234} x_3 x_4$,
$y_{134} x_3 x_4$, $y_{134} x_4 x_1$,
$y_{123} y_{134} x_1$, $y_{123} y_{134} x_3$,
$y_{412} y_{234} x_2$, $y_{412} y_{234} x_4$

\medskip

11 Triangles in $G_9'$: $xx_1x_2$, $xx_2 x_3$, $x x_3 x_4$, $x x_4 x_1$,
$y_{12} x_1 x_2$,
$y_{134} x_3 x_4$, $y_{134} x_4 x_1$,
$y_{234} x_2 x_3$, $y_{234} x_3 x_4$,
$y_{12} y_{134} x_1$,
$y_{12} y_{234} x_2$

\medskip

12 Triangles in $G_{10}'$: $xx_1x_2$, $xx_2 x_3$, $x x_3 x_4$, $x x_4 x_1$,
$y_{12} x_1 x_2$,
$y_{23}x_2x_3$,
$y_{134} x_3 x_4$, $y_{134} x_4 x_1$,
$y_{12} y_{134} x_1$, 
$y_{23} y_{134} x_1$,
$y_{12} y_{23} x_2$, 
$y_{12} y_{23} y_{134}$

\medskip

22 Triangles in $G_{11}'$: $xx_1x_2$, $xx_2 x_3$, $x x_3 x_4$, $x x_4 x_1$,
$y_{12} x_1 x_2$, $y_{34} x_3 x_4$,
$y_{123} x_1 x_2$, $y_{123} x_2 x_3$,
$y_{412} x_4 x_1$, $y_{412} x_1 x_2$,
$y_{234} x_2 x_3$, $y_{234} x_3 x_4$,
$y_{341} x_3 x_4$, $y_{341} x_4 x_1$,
$y_{12} y_{134} x_1$, $y_{12}y_{234} x_2$,
$y_{34} y_{124} x_4$, $y_{34} y_{123} x_3$,
$y_{134} y_{123} x_1$, $y_{134} y_{123}x_3$,
$y_{234} y_{412} x_2$, $y_{234} y_{412} x_4$

\medskip

15 Triangles in $G_{12}'$:  $xx_1x_2$, $xx_2 x_3$, $x x_3 x_4$, $x x_4 x_1$,
$y_{12} x_1 x_2$, $y_{23} x_2 x_3$,
$y_{34} x_3 x_4$, 
$y_{124} x_1 x_2$, $y_{124} x_4 x_1$, 
$y_{12} y_{23} x_2$, $y_{23} y_{34} x_3$, 
$y_{124} y_{23} x_2$, $y_{124} y_{34} x_4$,
$y_{12} y_{23} y_{34}$,
$y_{124} y_{23} y_{34}$

\end{document}